\documentclass[a4paper]{article}
\usepackage[latin1]{inputenc}
\usepackage{fullpage}
\usepackage{color}
\usepackage{epsfig}
\usepackage{subfig}
\usepackage{amssymb,tikz}

\usepackage{color}

\usepackage{graphicx}
\usepackage{amsmath,amssymb,amsthm}
\usepackage{algorithm2e}
\usepackage{hyperref}

\usepackage{thmtools}
\usepackage{thm-restate}
\usepackage{cleveref}

\newtheorem{theorem}{Theorem}

\newtheorem{definition}{Definition}

\newtheorem{corollary}{Corollary}
\newtheorem{lemma}{Lemma}
\newtheorem{claim}{Claim}

\usepackage[framemethod=tikz]{mdframed}

\newcommand{\qedclaim}{\hfill $\diamond$ \medskip}
\newenvironment{proofclaim}{\noindent{\em Proof of the claim.}}{\qedclaim}

\title{Equitable orientations of sparse uniform hypergraphs}
\author{Nathann Cohen\thanks{CNRS and Université Paris-Sud -- nathann.cohen@gmail.com}, William Lochet\thanks{Univ. Nice Sophia Antipolis, CNRS, I3S, INRIA, LIP, ENS de Lyon -- william.lochet@gmail.com}}

\begin{document}

\maketitle
\abstract{Caro, West, and Yuster studied how $r$-uniform hypergraphs can be
  oriented in such a way that (generalizations of) indegree and outdegree are as
  close to each other as can be hoped. They conjectured an existence result of
  such orientations for sparse hypergraphs, of which we present a proof.}
%\tableofcontents

\section{Introduction}

In~\cite{Yuster}, Caro, West, and Yuster presented a generalization to
hypergraphs of the notion of {\em orientation} defined for graphs. Their
acknowledged purpose is to study how hypergraphs can be oriented in such a way
that minimum and maximum degree are close to each other, knowing that reaching
an additive difference of $\leq 1$ is always achievable in the case of graphs.
Identifying an {\em orientation} of an edge with a total ordering of its
elements, they define a notion of degree on oriented $r$-uniform hypergraphs.

\begin{definition}
  Let $\mathcal H$ be a $r$-uniform hypergraph, and let every $S\in\mathcal H$
  define a total order on its elements as a bijection $\sigma_S:S\mapsto
  [r]$. The degree $d_P(U)$ of a set of vertices $U\subseteq V(\mathcal H)$ with
  respect to a set of positions $P\subseteq [r]$ (where $|U|=|P|$) is equal
  to: $$d_P(U)=|\{S\in\mathcal H:U\subseteq S\text{ and }\sigma_S(U)=P\}|$$
\end{definition}

\noindent From there they define {\em equitable} orientations:

\begin{definition}
  The orientation of a $r$-uniform hypergraph $\mathcal H$ is said to be
  $p$-equitable if $|d_{P}(U)-d_{P'}(U)|\leq 1$ for any choice of $U\subseteq
  V(\mathcal H)$ and $P,P'\subseteq [r]$ of cardinality $p$. It is said to be
  {\em nearly $p$-equitable} if the looser requirement $|d_{P}(U)-d_{P'}(U)\leq
  2|$ holds.
\end{definition}

They gave proof that all hypergraphs admit a 1-equitable as well as a
$(r-1)$-equitable orientation, and also proved that some hypergraphs do not
admit a $p$-equitable orientation for all values of $p$. For a fixed value of
$p$ and $k$, they proved the existence of $r_0(p,k)$ such that every $r$-uniform
hypergraph $\mathcal H$ with $r\geq r_0(p,k)$ admits a nearly $p$-equitable
orientation whenever it is sufficiently sparse, i.e.: $$\Delta_p(\mathcal
H)=\max_{{U\subseteq V(\mathcal H)}\atop |U|=p}|\{S\in\mathcal H:U\subseteq
S\}|\leq k$$ They conjectured that a $p$-equitable orientation actually exists,
which we prove here.

\begin{theorem}
  Let $p,k$ be fixed integers. There exists $r_0$ such that for every
  $r\geq r_0$, every $r$-uniform hypergraph with $\Delta_p(\mathcal H)\leq k$
  admits a $p$-equitable orientation.
\end{theorem}
Note that, in the case where $r$ is big compared to $\Delta_p(\mathcal H)$, a
$p$-equitable orientation means that $d_P(U)$ is equal to $0$ or $1$ for every
choice of set of positions $P$ and set of vertices $U$.\\

In order to prove the existence of nearly $p$-equitable orientation, Caro, West,
and Yuster~\cite{Yuster} used the {L}ov\'asz Local Lemma. In \cite{Moser}, Möser
and Tardos presented an elegant algorithmic proof of it which developed the
technique of entropy compression. Our proof uses that technique and the
following Lemma (proved in Section~\ref{sec:derangements}) that counts what can
be seen as a generalization of derangements.

\begin{restatable*}{lemma}{lemmaderangements}
  \label{lem:derangement}
  Let $p,k\in \mathbb N$ and $\alpha<1$ be fixed. Let $X$ be a set of
  cardinality $r$ and let $\mathcal L_S$ be, for every $S\in \binom {X}{p}$, a
  collection of $p$-subsets of $X$ with $|\mathcal L_S|\leq k$. Then, if no
  $p$-subset occurs in more than $r^\alpha$ of the $\mathcal L_S$, a random
  permutation $\sigma$ of $X$ satisfies $\sigma(S)\not\in \mathcal L_S$ for
  every $S$ with probability $\geq (1-2k/\binom r p)^{\binom r p}=
  e^{-2k}+o(1)$ when $r$ grows large.
\end{restatable*}

%%%%%%%%%%%%%%%%%%%%%%

\section{Algorithm}

In what follows, we assume that every finite set $S$ has an implicit enumeration
on its elements, and in particular that the edges of a hypergraph $\mathcal H$
are implicitly ordered. We will say that $i$ \textit{represents} an
element $s\in S$ when $s$ is the $i$-th element of $S$ in this implicit ordering.\\

We will orient the edges of $\mathcal H$ one by one as a (partial) equitable
orientation of $\mathcal H$, i.e. in such a way that any $p$-subset of $V$ never
appears more than once at the same position among the oriented edges. To do so,
we require the partial orientation to enforce an additional property.

\begin{definition}
  Let $\mathcal H$ be a partially oriented $r$-uniform hypergraph. We say that
  an edge $S\in\mathcal H$ is {\em pressured} by the (oriented) edges
  $S_1,\dots,S_l$ if there exists $P\in \binom {[r]}{p}$ such that for every $i$
  the set $S_i$ attributes the positions of $P$ to some $p$-set
  $s_i\subseteq S_i\cap S$.
\end{definition}

Note that Lemma~\ref{lem:derangement} ensures that a partial orientation of
$\mathcal H$ can be extended to an unoriented edge $S$, provided that no family
of more than $r^\alpha$ oriented edges pressures $S$. It asserts, for $c < e^{-2k}$ and $r$ sufficiently large, that at least
$cr!$ orientations of $S$ are admissible for this extension: we name them
{\it good} permutations of $S$. Algorithm~\ref{algo:entropy:nondet} selects an ordering randomly among
them, while ensuring that no other edge is pressured by a family of edges larger
than $r_1=\lfloor r^{\alpha}\rfloor$.\\

\begin{algorithm}[H]
  \KwData{A $r$-uniform hypergraph $\mathcal H$ with $\Delta_p(\mathcal H)\leq k$}
  \KwResult{A $p$-equitable orientation of $\mathcal H$}\vspace{1mm}
  \While{not all edges are oriented}{
    $S_1\gets$ unoriented edge of smallest index\\
    Pick for $S_1$ the orientation indexed $v_{i}$ (among $\geq cr!$ available)\\
    \uIf{some edge $S$ of $\mathcal H$ is pressured by sets $S_1,\dots,S_{r_1}$}{
      Cancel the orientation of all edges $S_i$.
    }
  }
  Return the oriented $\mathcal H$
  \label{algo:entropy:nondet}
  \caption{A non-deterministic algorithm}
\end{algorithm}

\vspace{2mm}

Algorithm~\ref{algo:entropy:nondet} starts with every edge being unoriented.  At
each step it orients the unoriented edge of smallest index by choosing a random
permutation amongst the $cr!$ first good permutations. We call {\em bad event}
the event that an edge $S\in\mathcal{H}$ is pressured by $\geq r_1$ other edges
$S_1, \dots, S_{r_1}$. If a bad event occurs after orienting $S_1$, then the
algorithm erases the orientation of the $S_1, \dots, S_{r_1}$.
%\nathann{Et aussi
%  l'orientation de $S$ non? C'est ce que j'ai mis dans les algorithmes.}
%\nathann{Techniquement y'a aussi la possibilité que plusieurs bad events
%  arrivent en meme temps, non ? Auquel cas j'imagine qu'on n'en fixe qu'un
%  seul.}\william{J'ai changé l'algo, techniquement on oriente $S_1$ (enfin une des aretes qui pressure).
%  Du coup si plusieurs Bad events arrivent en meme temps, comme $S_1$ est responsable et que tu le supprime
%  ca pose pas de pb}

It is trivial to see that Algorithm~\ref{algo:entropy:nondet} only returns $p$-equitable orientations of $\mathcal H$.
Moreover, every time the algorithm chooses a random permutation, it does so among at least $cr!$ good ones by Lemma~\ref{lem:derangement}.
%\nathann{Cette derniere phrase ne me parait pas apporter beaucoup.}. \william{On fait un choix random parmis cr!, autant dire
%qu'il y a effectivement cr! a chaque étape}
Note that we need to consider large families pressuring already oriented edges: indeed, we
might have to cancel the orientation of such an edge to redefine it again later.

\begin{theorem}\label{AlgoTerm}
  Let $p,k\in\mathbb N$, $\alpha,c\in\mathbb R^*_+$ with $\alpha<1$ and $c<e^{-2k}$. For every sufficiently large $r$, there is a set of random choices for which Algorithm~\ref{algo:entropy:nondet} terminates.
\end{theorem}

In order to prove this we will analyse the
possible executions of the $M$ first steps of
Algorithm~\ref{algo:entropy:nondet}. To do this we make it deterministic and obtain
Algorithm~\ref{algo:entropy:det}, in the following way:

\begin{itemize}
\item Take as input a vector $v\in [cr!]^M$ which simulates the random choices.
\item Output a log (or {\it trace}) when it is not able to orient all edges.
\end{itemize}

\noindent We define a \textit{log of order $M$} to be a triple $(R,X,F)$ where:

\begin{itemize}
	\item $R$ is a binary word whose length lies between $M$ and $2M$.%, where sequences of consecutive 0 have length $r_1$.
	\item $X$ is a sequence of $h$ 7-tuples of integers $(x_1,x_2,x_3,x_4,x_5, x_6, x_7)$ where:\\[-5mm]
          \begin{center}
          \begin{tabular}{llll}
            $x_1 \leq  \binom{r}{p}$& $x_2 \leq k$& $x_3 \leq \binom{\binom{r}{p}}{r_1-1}$& $x_4 \leq k^{r_1 - 1}$\\[2mm]
            $x_5 \leq p!^{r_1-1}$& $x_6 \leq (r-p)!^{r_1-1}$& $x_7 \leq r!$&\\
          \end{tabular}
          \end{center}

	\item $F$ is an integer smaller than $(r!+1)^{\mid \mathcal{H} \mid}$ representing a partial orientation of $\mathcal H$.
\end{itemize}

The {\em log} is actually a {\em trace} of the deterministic algorithm's execution.
Its objective is to encode which orientations get canceled during the algorithm's execution.
%We will show later that there is an injection from the set of vector $v\in [cr!]^M$ for which Algorithm~\ref{algo:entropy:det} produces a log to the set of log.
We will show later that Algorithm~\ref{algo:entropy:det} cannot produce the same
log from two different input vectors $v,v'\in [cr!]^M$.
and that, for $M$ big enough, that the set of possible log
is smaller than $(cr!)^M$.
We now describe the log and how Algorithm~\ref{algo:entropy:det} produces it.

\begin{itemize}
        \item $R$ is initialized to the empty word. We append $1$ to $R$ whenever Algorithm~\ref{algo:entropy:det} adds a new orientation; we append 0 whenever it cancels one.
	\item Consider the following bad event: after orienting $S_1$, an edge $S\in\mathcal H$ is pressured by $r_1$ other edges $S_1, \dots, S_{r_1}$.
	We note $s_i$ the set of vertices that $S_i$ maps to $P$. We associate the following 7-tuple which identifies the sets $S_i$ as well as their orientation:
	\begin{itemize}
	 	\item $x_1< \binom r p$ represents the set $s_1$  among the $\binom{r}{p}$ possible subsets of size $p$ of $S_1$.
		\item $x_2< k$ identifies $S$ as one of the (at most $k$) edges containing $s_1$.
                  %We know that $N(s_1)$ is smaller than $k$, so let $x_2$ be the integer in $[k]$ representing $S$ in $N(s_1)$.
		\item $x_3< \binom {\binom r p} {r_1-1}$ is an integer representing the set of subsets $s_2, \dots s_{r_1}$  amongst the $\binom{r}{p}$ subsets of size $p$ of $S$.
		\item $x_4< k^{r_1 - 1}$ is an integer representing the sequence $(y_2, \dots, y_{r_1})\in [k]^{r_1-1}$ such that the $y_l$-th edge containing $s_l$ is $S_l$.
		\item $x_5< p!^{r_1-1}$ is an integer representing the sequence $(p_1, \dots, p_{r_1})$, where $p_i \in [p!]$ represents the subpermutation of $S_i$ onto $s_i$ (we know it's a permutation of $P$).
		\item $x_6< (r-p)!^{r_1-1}$ is the integer representing the sequence  $[p_2, \dots, p_{r_1}]$, where $p_i \in [(r-p)!]$ represent the subpermutation of $S_i$ onto $[r] \setminus s_i$.
		\item $x_7< r!$ is the integer representing the permutation chosen for $S_1$.

	\end{itemize}
	% Algorithm~\ref{algo:entropy:det} append to $X$ this 7-tuple.
        $X$ is the list of the $7$-tuples describing the bad events, in the
        order in which they happen.

      \item $F$ is the integer representing the partial orientation of $\mathcal
        H$ (i.e. a choice among $r!+1$ per edge of $\mathcal H$) when
        Algorithm~\ref{algo:entropy:det} returns.

\end{itemize}

This gives the following Algorithm 2:

\begin{algorithm}[H]
  \KwData{
    \begin{enumerate}
    \item A $r$-uniform hypergraph $\mathcal H$ with $\Delta_p(\mathcal H)\leq k$,\\\vspace{-2mm}
    \item A vector $v\in [cr!]^M$
    \end{enumerate}
    }
  \KwResult{A $p$-equitable orientation of $\mathcal H$, or a log of order $M$}\vspace{1mm}
  $R\gets\emptyset, X\gets\emptyset$\\
  \For{$1\leq i \leq M$}{
    $S_1\gets$ unoriented edge of smallest index\\
    Pick for $S$ the orientation indexed $v_{i}$ among $\geq cr!$ available\\
    \uIf{some edge of $\mathcal H$ is pressured by sets $S_1,\dots,S_{r_1}$}{
      Append $1$ to the end of $R$\\
      Append to $X$ a 7-tuple describing the conflict\\
      Cancel the orientation of all $r_1+1$ edges involved in the conflict
    }
    \uElseIf{all edges are oriented}{
      Return the oriented $\mathcal H$
    }
    \Else{
      Append $0$ to the end of $R$
    }
  }
  $ F \gets$ the integer representing the partial orientation of $\mathcal{H}$. \\
  Return $(R,X,F)$
  \label{algo:entropy:det}
  \caption{A deterministic algorithm}
\end{algorithm}
\vspace{2mm}

 We will show the following claim.

\begin{claim}
Let $e$ be a vector in $[cr!]^M$ from which Algorithm~\ref{algo:entropy:det} cannot produce a $p$-equitable orientation of $\mathcal H$ and outputs a log $(R,X,F)$.
We can reconstruct $e$ from $(R,X,F)$.
\end{claim}

\begin{proofclaim}
First we show that we can find for every $z\leq M$, the set $C(z)$ of edges for which a orientation after $z$ steps.
We proceed by induction on $z$, starting from $C(0)=\emptyset$. At step $z+1$, Algorithm~\ref{algo:entropy:det} chooses a orientation for the smallest index $i$
not in $C(z)$. If, in $R$, the $(z+1)$-th 1 is not followed by a 0, then there is no bad event triggered by this step. In this case the set
$C(z+1)$ is the set $C(z) \cup {i}$. Suppose now that the $(z+1)$-th 1 is followed by a sequence of 0: this means that the
algorithm encountered a bad event. By looking at the number of sequences of 0 in $R$ before the $z+1$-th 1 we can deduce the
number of bad events before this one. This mean we can find, in $X$, the 7-tuple $(x_1, x_2, x_3, x_4, x_5, x_6, x_7)$ associated to
this bad event. We take the following notations for the bad event : After orienting $S_1$, an edge $S$ of $H$ is pressured by $r_1$ other edges, $S_1, \dots S_{r_1}$.
We note $s_i$ the subset of $S_i$ that are sent to $P$. $S_1$ is the last edge we oriented (known by induction), $x_1$ indicates
$s_1$ amongst the subset of $S$, $x_2$ indicates $S$ amongst the set of edges containing $s_1$, $x_3$ indicates the $s_d$ for $d \in [2..r_1]$,
and $x_4$ indicates the $S_d$ for $d \in [2..r_1]$. In this case the set $C(z+1)$ is the set $C(z)$ for which we removed
all the $ES_d$ for $d \in [2..r_1]$.

We can now deduce the set $S(z)$ of all chosen orientations after $z$ steps.
We also proceed by induction, this time starting from step $M$. By construction, $F$ is exactly the integer
representing the partial orientation of $\mathcal H$ at step $M$. If the last letter of $R$ is a $1$, this means the last step of the
algorithm consisted only of choice of a orientation. We just showed that we know which orientation was chosen after $M-1$
steps, so we can deduce the state of all orientation after $M-1$ steps. If the last letter is a 0, Algorithm~\ref{algo:entropy:det}
encountered a bad event. Keeping the notation of the bad event, let $(x_1, x_2, x_3, x_4, x_5, x_6, x_7)$ be the 7-tuple associated to this bad event.
Like before	$x_1$, $x_2$, $x_3$, $x_4$ and the knowledge of $C(M-1)$ allow us to know which permutations Algorithm~\ref{algo:entropy:det} erased at this
step. Moreover $x_7$ tells us the random choice made by Algorithm~\ref{algo:entropy:det} and from $x_7$ and $x_1$ we can deduce $P$. For each $s_i$
we know the orientation chosen for $S_i$ at the step $M-1$ sends $P$ onto $s_i$, from $x_5$ we deduce exactly in which order and from
$x_6$ we get the rest of the orientation. Therefore we can deduce the set of chosen orientations before the bad event occurred.
With the sets $S(z)$ and $C(z)$ known for all $z \leq M$ we can easily deduce $e$.
\end{proofclaim}

\noindent The previous claim has the following corollary:
\begin{corollary}
  If $\mathcal H$ admits no $p$-equitable orientation, then Algorithm~\ref{algo:entropy:det} defines an injection from the set of vectors $[cr!]^M$ into $L^M$.
\end{corollary}

Let $L_M$ be the set of all possible logs after $M$ steps of Algorithm 2.
To show Theorem \ref{AlgoTerm} it suffices to show that, for $M$ big enough, $| L_M |$
is strictly smaller than $(cr!)^{M}$.

\begin{lemma}
For $M$ big enough,  $| L_M | < (cr!)^{M}$.
\end{lemma}

\begin{proof}
We will compute a bound for $| L_M |$.
$R$ is a binary word of size $\leq 2M$, and there are at most $4^M$ such words.
$X$ is a list of 7-tuples. As Algorithm 2 made $M$ choices and each bad event removes $r_1$ of those,
there exist at most $\frac{M}{r_1}$ bad events. Moreover, for each 7-tuple,
$(x_1,x_2,x_3,x_4,x_5, x_6, x_7)$ we have $x_1 \leq  \binom{r}{p}$, $x_2 \leq k$,
$x_3 \leq \binom{\binom{r}{p}}{r_1-1}$, $x_4 \leq k^{r_1 - 1}$, $x_5 \leq p!^{r_1-1}$, $x_6 \leq (r-p)!^{r_1-1}$, $x_7 \leq r!$.
Using the bounds $\binom{n}{k} \leq (\frac{n\cdot e}{k})^k$ or $\binom{n}{k} \leq n^k$ we get the following bound.
% \nathann{
% \begin{align*}
% |X|&\leq \left[ r^p \cdot k \cdot \left(\frac{r^p \cdot e}{r_1-1}\right)^{r_1 - 1} \cdot k^{r_1 - 1 } \cdot p!^{r_1 - 1 } \cdot (r-p)!^{r_1 - 1 } \cdot r! \right]^{M/r_1}\\[-2mm]
% \intertext{Then, assuming a sufficiently large $r$ with respect to $p,k$:}\\[-6mm]
% |X|&\leq O(r^p)^{(1-\epsilon) M} \cdot O(1)^{(1-\epsilon) M}\cdot \left(\frac {O(r^{p})}{r_1-1}\right)^{\epsilon M}\cdot O(1)^{\epsilon M}\cdot O(1)^{\epsilon M} (r-p)!^{\epsilon M}\cdot r!^{(1-\epsilon) M}\\
% &\leq O(r^p)^{(1-\epsilon) M} \cdot \left(\frac {O(r!)}{r_1-1}\right)^{\epsilon M}\cdot r!^{(1-\epsilon) M}\\
% &\leq \frac {O\left(r^{(1-\epsilon)p}\right)^{ M}}{(r_1-1)^{\epsilon M}}\cdot O(r!)^{M}\\
% \end{align*}
% This is equal to $o(r!)^M$ whenever $\alpha \epsilon>(1-\epsilon)p$, i.e. when
% }

\begin{align*}
|X| &\leq \left( r^p \cdot k \cdot \left(\frac{r^p \cdot e}{r_1-1}\right)^{r_1 - 1} \cdot (k \cdot p! \cdot (r-p)!)^{r_1 - 1 } \cdot r! \right)^{M/r_1}\\
&\leq \frac{\left(r! \cdot (r^p)^{r_1} \cdot (r-p)!^{r_1 - 1} \right)^{M/r_1} \cdot (k \cdot e \cdot p!)^{M}}{(r_1-1)^{M(r_1-1)/r_1}}\\
& \leq \left[r^p \cdot r!^{r_1} \cdot \left(\frac{r^p}{r(r-1)\dots (r-p +1)}\right)^{r_1-1} \right]^{M/r_1}
\cdot \left( \frac{k \cdot e \cdot p!}{(r_1-1)^{(r_1-1)/r_1}} \right)^M\\
\intertext{We can assume $r  > 2p$, and so $\frac{r}{r-p+1}< 2$:}
& \leq r!^M \cdot \left( r^{p/r_1} \cdot 2^{p} \cdot\frac{k \cdot e \cdot p!}{(r_1-1)^{(r_1-1)/r_1}} \right)^M
\end{align*}
As $F \leq (r!+1)^{| \mathcal{H} |}$, we get the following bound on $| L_M |$:

$$ | L_M | \leq r!^M \cdot \left(4 \cdot r^{p/r_1} \cdot 2^{p} \cdot\frac{k \cdot e \cdot p!}{(r_1-1)^{(r_1-1)/r_1}} \right)^M \cdot (r!+1)^{| \mathcal{H} |}  $$

\end{proof}

%%%%%%%%%%%%%%%%%%%%%%

\section{Derangements}
\label{sec:derangements}

The results of this section are based on a lemma from Erd{\H{o}}s and
Spencer~\cite{MR1095369}:

\begin{lemma}[Lopsided {L}ov\'asz Local Lemma]
  \label{llll}
  Let $A_1,\dots,A_m$ be events in a probability space, each with probability at
  most $p$. Let $G$ be a graph defined on those events such that for every
  $A_i$, and for every set $S$ avoiding both $A_i$ and its neighbours, the
  following relation holds:
  $$P[A_i|\bigwedge_{A_j\in S}\bar A_j] \leq P[A_i]$$
  Then if $4dp\leq 1$, all the events can be avoided simultaneously:
  $$P[\bar A_1\wedge \dots \wedge \bar A_m]\geq (1-2p)^m> 0$$
\end{lemma}

Thanks to this result we can prove the following, which can be seen as a
generalization of the fact that a random permutation of $n$ points is a
derangement with asymptotic probability $n!/e$.

\lemmaderangements
\begin{proof}
  For every $S\in \binom X p$, we define the {\it bad event} $B_S$ with:
  $$B_S = \bigvee_{S'\in \mathcal L_S}[\sigma(S) = S']$$
  Each $B_S$ has a probability $P[B_S]\leq k/\binom r p $. On these bad events we define a lopsidependency graph (see \cite{MR1095369})
  $G_B$ with the following adjacencies:
  $$\left\{(S_1,S_2): S_1,S_2\in \binom X p\text{ s.t. }\Big[ S_1\bigcup \mathcal L_{S_1}\Big]\bigcap \Big[ S_2\bigcup \mathcal L_{S_2}\Big] \neq \emptyset\right\}$$

  As a $p$-subset of $X$ intersects at most $O(r^{p-1})$ others, and noting that
  every $p$-subset can occur at most $r^\alpha$ times, we have that:

  $$\Delta(G_B)\leq (k+1)r^\alpha\times O(r^{p-1}) = o(r^p)$$

  In order to apply the Lopsided {L}ov\'asz Local Lemma to the events $B_S$ and
  graph $G_B$, we must ensure for every $S\in \binom X p$ and $S_B\subseteq
  V(G_B)\backslash N_{G_B}[B_S]$ that:
  \begin{equation}
    \label{eq:bound}
    P(B_S|\bigwedge_{B_{S'}\in S_B}\bar B_{S'})\leq P(B_S)
  \end{equation}
  Indeed, if we denote by $T$ (for {\it trace}) the number of elements of
  $\bigcup_{B_{S'}\in S_B} S'$ sent by the random permutation $\sigma$ into
  $\bigcup \mathcal L_S$:
  \begin{align*}
    P(B_S) &=\sum_t P(B_S\mid T=t)P(T=t)\\
    P(B_S|\bigwedge_{B_{S'}\in S_B}\bar B_{S'}) &= \sum_t P(B_S\mid T=t,\bigwedge_{B_{S'}\in S_B}\bar B_{S'})P(T=t\mid\bigwedge_{B_{S'}\in S_B}\bar B_{S'})\\
    \intertext{As $\bigcup \mathcal L_S$ is disjoint from the $\bigcup \mathcal L_{S'},\forall B_{S'}\in S_B$, we have $P(B_S\mid T=t,\bigwedge_{B_{S'}\in S_B}\bar B_{S'})=P(B_S\mid T=t)$ and:}\\[-7mm]
           &= \sum_t P(B_S\mid T=t)P(T=t\mid\bigwedge_{B_{S'}\in S_B}\bar B_{S'})
  \end{align*}
  In order to prove (\ref{eq:bound}), we will first need the following observation:

  \begin{claim}
    \label{claim:decrease}
    $P(B_S\mid T=t)$ is a decreasing function of $t$.
  \end{claim}
  \begin{proofclaim}
    We compute the value of $P(B_S\mid T=t)$ exactly, denoting by $r'\leq r$ the
    cardinality of $\bigcup_{B_{S'}\in S_B} S'$. It is equal to 0 when $t>r'-p$,
    and is otherwise equal to:
    % (The first equality is because the events are disjoint)
    \begin{align*}
      P(B_S\mid T=t)&=\sum_{S'\in\mathcal L_S}P(\sigma(S)=S'\mid T=t)\\
      &=\frac {|\mathcal L_S|} {\binom {r-t} p}\frac {\binom {r'-p} t} {\binom {r'} t}\\
      &=\left(|\mathcal L_S|\frac {(r'-p)!p!} {r'!}\right)\left(\frac {(r-p-t)!(r'-t)!} {(r'-p-t)!(r-t)!}\right)\\
      &=P(B_S\mid T=t-1)\left(\frac {(r'-p-t+1)} {(r-p-t+1)}\frac {(r-t+1)} {(r'-t+1)}\right)\\
      &\leq P(B_S\mid T=t-1)\\[-12mm]
    \end{align*}
  \end{proofclaim}

  Additionally, we will prove a relationship on the members of $\sum_t P(T=t)$ and on those of $\sum_t P(T=t\mid\bigwedge_{B_{S'}\in S_B}\bar B_{S'})$, which both sum to~1:

  \begin{claim}
    \label{claim:ratio}
    If $P(T=t\mid\bigwedge_{B_{S'}\in S_B}\bar B_{S'})$ is nonzero, then $$\frac {P(T=t+1)} {P(T=t)} \leq \frac {P(T=t+1|\bigwedge_{B_{S'}\in S_B}\bar B_{S'})} {P(T=t\mid\bigwedge_{B_{S'}\in S_B}\bar B_{S'})}$$
  \end{claim}
  \begin{proofclaim}
    According to Bayes' Theorem applied to the right side of the equation,
    $$\frac {P(T=t+1|\bigwedge_{B_{S'}\in S_B}\bar B_{S'})} {P(T=t\mid\bigwedge_{B_{S'}\in S_B}\bar B_{S'})} = \frac {P(\bigwedge_{B_{S'}\in S_B}\bar B_{S'}\mid T=t+1)P(T=t+1)} {P(\bigwedge_{B_{S'}\in S_B}\bar B_{S'}\mid T=t)P(T=t)}$$
    We thus only need to ensure the following, which is a consequence of
    Lemma~\ref{lem:decreasingt}:
    $${P(\bigwedge_{B_{S'}\in S_B}\bar B_{S'}\mid T=t+1)} \geq {P(\bigwedge_{B_{S'}\in S_B}\bar B_{S'}\mid T=t)}$$
  \end{proofclaim}

  \noindent We are now ready to prove (\ref{eq:bound}), and we define $d_t$ for every $t$ where $P(T=t)$ is nonzero:

  $$d_t=P(T=t)-P(T=t\mid\bigwedge_{B_{S'}\in S_B}\bar B_{S'})$$

  \noindent By definition the sum $\sum_t d_t$ is null, and in order to prove
  (\ref{eq:bound}) we need only show that the sum $\sum_t d_t P(B_S\mid T=t)$ is
  nonnegative. It is a consequence of Claim~\ref{claim:ratio} that all
  nonnegative values of $d_t$ appear before all nonpositive ones, and so that
  there is a $t_0$ such that $d_t\geq 0$ iff $t\leq t_0$. As a result,
  $|\sum_{t\leq t_0}d_t| = |\sum_{t> t_0}d_t| = \frac 12$ and we can write:
  \begin{align*}
    \sum_t d_t P(B_S\mid T=t) &= \sum_{t\leq t_0} d_t P(B_S\mid T=t) + \sum_{t> t_0} d_t P(B_S\mid T=t)\\
    &\geq \frac 1 2 P(B_S\mid T=t_0) - \frac 1 2 P(B_S\mid T=t_0+1)\geq 0&\text{(by Claim~\ref{claim:decrease})}
  \end{align*}
  The second hypothesis of Lemma~\ref{llll} is that $4pd\leq 1$, which
  translates in our case to $4\frac k {\binom rp}o(r^p)=o(1)$ and is thus
  satisfied when $r$ grows large. Hence, we have that:
  $$P[\bigwedge_S\bar B_S]\geq \Big[1-2k/\binom rp\Big]^{\binom rp}=e^{-2k}+o(1)$$
\end{proof}

\begin{lemma}
  \label{lem:decreasingt}
  Let $A,B$ be two sets of size $r$, and let $\sigma:A\mapsto B$ be a random
  bijection. For every $A_1,\dots,A_k\subset A'\subset A$ and
  $B_1,\dots,B_k\subset B' \subset B$, the following function increases with $t$.
  \begin{align}
    \label{eq:shortcut}
    P\left[\bigwedge_i\left[\sigma(A_i)\neq B_i\right]\Bigm|\sigma(A')\backslash B'\text{ has cardinality }t\right]
  \end{align}
\end{lemma}
\begin{proof}
  We implicitly assume in this proof that the conditionning event has a nonzero
  probability for $t$ and $t+1$. Let $S_1,S_2$ be two sets of cardinality $|A'|$
  with symmetric difference $S_1\Delta S_{2}=\{x,y\}$ where $x\in S_2$ is an
  element of $B\backslash B'$. Let $\sigma_{xy}$ be the permutation transposing
  $x$ and $y$. Then,
  \begin{align*}
    P\left[\bigwedge_i[\sigma(A_i)\neq B_i]\Bigm| \sigma(A') = S_1\right] &\leq P\left[\bigwedge_i[\sigma_{xy}\sigma(A_i)\neq B_i]\Bigm| \sigma(A') = S_1\right]\\
                                                                          &=    P\left[\bigwedge_i[           \sigma(A_i)\neq B_i]\Bigm| \sigma(A') = S_2\right]\\
  \end{align*}
  We can use this inequality to derive the result:
  {
    \scriptsize
    \begin{align*}
      (\ref{eq:shortcut}) &= \frac {1} {\binom {|B\backslash B'|} {t} \binom {|B'|} {|A'|-t}} \sum_{{S\subseteq B\atop |S|=|A'|}\atop |S\backslash B'|=t} P\left[\bigwedge_i\left[\sigma(A_i)\neq B_i\right]\Bigm|\sigma(A')=S\right]\\
                          &\leq \frac {1} {\binom {|B\backslash B'|} {t} \binom {|B'|} {|A'|-t}} \sum_{{S\subseteq B\atop |S|=|A'|}\atop |S\backslash B'|=t} \frac {1} {(|B\backslash B'|-t)(|A'|-t)} \sum_{{{S'\subseteq B\atop |S'|=|A'|}\atop |S'\backslash B'|=t+1}\atop |S\Delta S'|=2} P\left[\bigwedge_i\left[\sigma(A_i)\neq B_i\right]\Bigm|\sigma(A')=S'\right]\\
                          &= \frac {1} {\binom {|B\backslash B'|} {t} \binom {|B'|} {|A'|-t}} \frac {(t+1)(|B'|-|A'|+t+1)} {(|B\backslash B'|-t)(|A'|-t)} \sum_{{{S'\subseteq B\atop |S'|=|A'|}\atop |S'\backslash B'|=t+1}} P\left[\bigwedge_i\left[\sigma(A_i)\neq B_i\right]\Bigm|\sigma(A')=S'\right]\\
                          &= \frac {\binom {|B\backslash B'|} {t+1} \binom {|B'|} {|A'|-t-1}} {\binom {|B\backslash B'|} {t} \binom {|B'|} {|A'|-t}} \frac {(t+1)(|B'|-|A'|+t+1)} {(|B\backslash B'|-t)(|A'|-t)} P\left[\bigwedge_i\left[\sigma(A_i)\neq B_i\right]\Bigm|\sigma(A')\backslash B'\text{ has cardinality }t+1\right]\\
                          &= P\left[\bigwedge_i\left[\sigma(A_i)\neq B_i\right]\Bigm|\sigma(A')\backslash B'\text{ has cardinality }t+1\right]
    \end{align*}
  }
\end{proof}

\bibliography{prod}
\bibliographystyle{plain}

\end{document}